\documentclass[12pt]{amsart}

\usepackage{amsmath,amsthm, amsfonts,amssymb, stmaryrd,yfonts,pxfonts,pifont, bbm}

 \usepackage{lineno}
 \usepackage{geometry}

\newcommand{\N}{\mathbb{N}}

\newcommand{\Z}{\mathbb{Z}}
\newcommand{\F}{\mathbb{F}}
\usepackage{mathrsfs}

 \usepackage{breqn}

\usepackage{hyperref}

\newtheorem{thm}{Theorem}[section]

\newtheorem{cor}{Corollary}[section]

\newtheorem{ex}{Example}[section]
\newtheorem{Remark}{Remark}[section]
\newtheorem{prop}{Proposition}[section]
\theoremstyle{definition}
\newtheorem{definition}{Definition}[section]

\newtheorem{remark}{Remark}[section]

\begin{document}
\begin{sloppypar}

\title{Proximal Relations in Asymptotically Commutative Non-Autonomous Dynamical Systems}
\author{%
Sushmita Yadav and Puneet Sharma
}

\address{Department of Mathematics, I.I.T. Jodhpur, N.H. 65, Nagaur Road, Karwar, Jodhpur 342037, INDIA}%
\email{yadav.34@iitj.ac.in, puneet@iitj.ac.in}%


\begin{abstract}

In this paper, we investigate various forms of proximal relations for non-autonomous dynamical systems. In particular, we introduce the notion of pointwise asymptotic commutativity and asymptotic commutativity, and show that these are strictly weaker than classical commutativity. We establish conditions under which different proximal relations, including proximality, syndetic proximality, and regional proximality, exhibit invariance and structural properties. In particular, we relate proximality with equicontinuity and weak mixing, and characterize conditions where proximal relations become trivial or maximal. Further, for systems generated by uniformly and strongly convergent sequences, we derive connections between proximal pairs of the system, its product systems, and the corresponding limiting system. Several examples are provided to demonstrate the necessity of the assumptions and the sharpness of the results. 

\end{abstract}

 \maketitle

\section{Introduction}

Dynamical systems provide a fundamental framework for understanding the long-term behavior of a wide range of natural and physical processes. Classical studies have largely focused on autonomous systems, where the governing rules remain fixed over time, leading to a well-developed theory with numerous applications. However, many real-world phenomena are inherently time-dependent, with underlying parameters that evolve, making the autonomous framework inadequate in capturing their full complexity. To capture such phenomena, Sergiy Kolyada and Lubomír Snoha introduced the notion of non-autonomous discrete dynamical systems in 1996 \cite{sk1}. Since then, there has been growing interest in studying both autonomous and non-autonomous systems, particularly in understanding the additional complexities that arise in time-dependent settings \cite{sk2, LC, SD, pm, pm2, SP}.\\

Proximality in dynamical systems, introduced by Ellis and Gottschalk \cite{EG}, has been central to understanding the qualitative behavior of systems through general group actions, where the authors derived sufficient conditions under which proximality and regional proximality coincide. In \cite{c}, syndetic proximality and regional syndetic proximality were introduced, and their relations with various forms of proximality were established. Notably, it was proved in \cite{SM} that an autonomous system is proximal if and only if it is syndetically proximal. Further, it is known that any minimal autonomous system $(X,f)$ is weakly mixing if and only if proximal pairs are dense in $X\times X$ \cite{kol}, and that a semiflow is equicontinuous if and only if it possesses no non-trivial regionally proximal pair \cite{DX}. If $(X,\F)$ is a sensitive system such that the proximal cell of every $x\in X$ is dense, then the system is Li-Yorke sensitive \cite{pm2}. For minimal non-autonomous systems, it was established in \cite{LC} that the proximal cell of each element is dense in $X$ if and only if proximal pairs are dense in $X\times X$. Recently, in \cite{SP}, authors proved that any proximal system $(X,\F)$ generated by a commutative family $\F$ has a fixed point as its unique minimal subset, and that any weakly mixing system $(X,\F)$ generated by a commutative family $\F$ has a dense proximal cell for every $x\in X$. The notions of proximality in dynamical systems have been used to investigate the qualitative behavior of various systems around us. While its conjunction with other known notions gives rise to strange (chaotic) behavior in a system, exploring proximality independently can provide better insight into the system, which in turn can help in conducting qualitative analysis of the system. This motivates us to introduce and study various notions of proximality for non-autonomous dynamical systems and investigate their interrelations and dynamical consequences.\\

The paper is organized as follows. In Section \ref{prelim}, we introduce the necessary preliminaries and basic definitions required for the subsequent sections. In particular, we define the notions of pointwise asymptotic commutativity and asymptotic commutativity, and provide examples to show that these notions are strictly weaker than commutativity. In Section \ref{pinv}, we establish a criterion under which the set of proximal pairs is invariant, and relate proximality to equicontinuity and weak mixing. In Section \ref{uniform}, we study non-autonomous systems generated by uniformly convergent sequences and relate the proximal pairs of the product system with those of its components. Finally, in Section \ref{strongc}, we consider non-autonomous systems generated by strongly convergent families and relate the proximal pairs of the system with those of the limiting system

\section {Preliminaries}\label{prelim}

Let $(X,d_X)$ be a compact metric space and $\F=(f_n)_{n\geq 1}$ be a sequence of surjective continuous self maps on $X$. It is known that for any initial state $x_0\in X$, the sequence $\F$ generates a non-autonomous system via the relation $x_{n}=f_{n}(x_{n-1})$. Let $(X,\F)$ denote the non-autonomous system generated by the sequence $\F$. For any initial state $x_0\in X$, let $\omega_{k, X}(x_0)=f_k\circ\ldots\circ f_1(x_0)$ denote the state of the system after $k$ instances of time. In case the underlying system $X$ is unambiguous, we denote the state of the system $\omega_{k,X}(x_0)$ by $\omega_k(x_0)$. Let $\mathcal{O}_\F(x)=\{\omega_{n}(x): n\in \mathbb{N}\}$ denote the \emph{orbit} of the point $x$ in $X$ under action of family $\F$. In case $f_i=f$ for all $i\in \N$, we denote orbit of $x\in X$ as $\mathcal{O}_f(x)$. For any $n,m\in \N$, let $f_{n}^{n+m}= f_{n+m}\circ f_{n+m-1}\circ\ldots\circ f_{n+1}$. In case the system is generated by a family of homeomorphisms, we consider the full orbit of any point $x$ in $X$ i.e., $\mathcal{O}_\F(x)=\{\omega_n(x): n\in \mathbb{Z}\}$ (where $\omega_{-n}(x)=(\omega_n)^{-1}(x)= f_1^{-1}\circ f_2^{-1}\circ\ldots f_n^{-1}(x)$ for all $n\in\N$). We say that the system $(X,\F)$ is a commutative if $f_i\circ f_j= f_j\circ f_i$ for all $i,j\in \mathbb{N}$.\\

A subset $S=\{s_n:n\in\N\}$ of $\N$ is \emph{syndetic} if there exists $M\in\N$ such that $|s_{n+1}-s_n|\leq M$ for all $n\in\N$. In such a case, we refer to $S$ as an \emph{$M$-syndetic} set. A subset $T\subseteq \N$ is said to be \emph{thick} if for any $k\in \N$ there exists $t\in \N$ such that $\{t+1,t+2,\ldots,t+k\}\subset T$. A set $S\subseteq \N$ is said to be \emph{thickly syndetic} if for any $k\in \N$, the set $\{t\in \N:\{t+1,t+2,\ldots,t+k\}\subset S\}$ is syndetic. A set $Y\subseteq X$ is said to be \emph{invariant} if $\mathcal{O}_\F(y)\subseteq Y$ for any $y\in Y$.

\begin{definition}\cite{pm2}
A nonautonomous system $(X,\mathbb{F})$ is said to be 
\emph{equicontinuous at a point} $x\in X$, if for each $\epsilon>0$, there exists $\delta_x>0$ such that $d_X(x, y)<\delta$ implies $d_X(\omega_n(x),\omega_n(y))<\epsilon$ for all $y\in X$, $n\in \N$. In case $\delta>0$ is independent to the choice of $x\in X$ then $(X,\F)$ is called \emph{equicontinuous nonautonomous system}.
\end{definition}

\begin{definition}\cite{pm2}
A nonautonomous system $(X,\mathbb{F})$ is said to be \emph{sensitive at a point} $x\in X$, if there exists $\delta_x>0$ such that for each neighborhood $U_x$ of $x$ there exists $k\in\N$ such that $diam(\omega_k(U_x))>\delta_x$.
In case $\delta>0$ is independent to the choice of $x\in X$ then $(X,\F)$ is called \emph{sensitive non autonomous system}.

\end{definition}



 \begin{definition}\cite{BP}
 A system $(X, \mathbb{F})$ is said to be \emph{weakly mixing} if for any collection of non-empty open sets $U_1, U_2, V_1, V_2$ in $X$, there exists a natural number $k$ such that $\omega_k(U_i)\cap V_i \neq\emptyset$, for $i=1,2$.
 \end{definition}

\begin{remark}[~\cite{pm}]
It is known that if the generating family $\mathbb{F}$ is commutative then $(X,\mathbb{F})$ is weakly mixing if and only if for each collection of non-empty open sets $U_1, U_2,\ldots, U_n$ and $V_1, V_2,\ldots, V_n$ in $X$ there exists $k\in\mathbb{N}$ such that $\omega_k(U_i)\cap V_i\neq\emptyset$ for $i=1,2,\ldots,n$.    
\end{remark} 

\begin{definition}\cite{LC} Let $(X,\F)$ be any non autonomous systems, then any pair $(x, y)$ is \emph{proximal} for $(X,\mathbb{F})$  if $\liminf\limits_{n} ~~d_X(\omega_n(x), \omega_n(y))=0$.
\end{definition}

\begin{definition} \cite{LC} 
Let $(X,\F)$ be any non autonomous systems, any pair $(x, y)$ is \emph{syndetically proximal} for $(X,\mathbb{F})$ if for any $\epsilon>0$, the set $\{n\in \N: d_X(\omega_n(x),\omega_n(y))<\epsilon\}$ is syndetic.
\end{definition}

We now introduce the notions of regionally proximal and regionally syndetical proximal pairs in non-autonomous dynamical system $(X,\F)$.

\begin{definition}
Let $(X,\F)$ be any non autonomous systems, then any pair $(x,y)$ is \emph{regionally proximal} for $(X,\F)$ if for any $\epsilon>0$ and neighborhoods $U,V$ of $x$ and $y$ respectively, there exists $z_1\in U$ and $z_2\in V$ such that that the set $\{n\in \N:d_X(\omega_{n}(z_1),\omega_{n}(z_2)<\epsilon\}$ is non-empty.
\end{definition}

\begin{definition}
Let $(X,\F)$ be any non autonomous systems, then any pair $(x,y)$ is \emph{regionally syndetically proximal} for $(X,\F)$ if for any $\epsilon>0$ and neighborhoods $U,V$ of $x$ and $y$ respectively, there exists $z_1\in U$ and $z_2\in V$ such that that the set $\{n\in \N:d_X(\omega_{n}(z_1),\omega_{n}(z_2)<\epsilon\}$ is syndetic.
\end{definition}

Let $P(X,\F)$, $L(X,\F)$, $Q(X,\F)$, and  $M(X,\F)$ denote the set of proximal, syndetically proximal, regionally proximal pairs, and regionally syndetically proximal respectively for $(X,\mathbb{F})$. The system $(X,\F)$ is said to be proximal (syndetically proximal/regionally proximal/regionally syndetically proximal) if the set $P(X,\F)$ ( $L(X,\F)$ / $Q(X,\F)$ / $M(X,\F)$ ) coincides with $X\times X$. For any $A,B \subset X\times X$, let $AB = A\circ B= \{(x,y): (x,z)\in A~ \text{,}~ (z,y)\in B ~~\text{for some}~~ z\in X\}$. 

Let $C(X)$ denote the collection of all continuous self maps on $X$ and $D_{C(X)}(f,g) = \sup \limits_{x\in X} d_X(f(x), g(x))$ (for any $f,g\in C(X))$. It is known that $D_{C(X)}$ defines a metric on $C(X)$ (known as the supremum metric). Further, any sequence $(f_n)$ in $C(X)$ converges uniformly to $f$ (in $C(X)$) if and only if it converges in $(C(X),D_{C(X)})$. We say that a sequence $(f_n)$ converges strongly to $f$ if $\sum \limits_{n=1}^{\infty} D_{C(X)}(f_n,f)<\infty$. It can be seen that every strongly convergent sequence is uniformly convergent (but the converse is not true). \\

It may be noted that in case the generating maps $f_n$'s coincide, the definitions coincide with the known notions for an autonomous dynamical system \cite{bc, bs, DV}. Some basic concepts and recent works in non-autonomous dynamical systems can be found in the literature \cite{kol, GD,sk1,sk2}.\\

The notion of commutativity has been extensively used in the literature to investigate the dynamical behavior of non-autonomous systems (see \cite{KB, SD, SP}). However, exact commutativity is often too restrictive and fails even for simple families of maps arising in applications. For example, two affine maps of the form $f(x)=ax+b$ and $g(x)=cx+d$ on a compact interval do not commute in general. On the other hand, many non-autonomous systems exhibit an approximate commutative behavior asymptotically, especially when the generating sequence converges to a limiting map or when the iterates approach the identity map in the long term.

Motivated by this observation, we introduce the notions of \emph{pointwise asymptotic commutativity} and \emph{asymptotic commutativity}, which are strictly weaker than commutativity and are satisfied by a much broader class of systems. 

\begin{definition} A non autonomous system $(X,\F)$ is said to be
\begin{enumerate}
    \item \emph{pointwise asymptotically commutative} if $\lim \limits_{n\rightarrow \infty}d_X(\omega_n\circ \omega_r(x),\omega_r\circ\omega_n(x))=0$ for every $x\in X$ and $r\in \N$.
    \item \emph{asymptotically commutative} if $\lim \limits_{n\rightarrow \infty}D_{C(X)}(\omega_n\circ \omega_r,\omega_r\circ\omega_n)=0$ for every $r\in \N$.
\end{enumerate}
    
\end{definition}

\begin{remark}
 It may be noted that for any non-autonomous system $(X,\F)$,
 \begin{enumerate}
     \item $(X,\F)$ is commutative
     \item $(X,\F)$ is asymptotically commutative
     \item $(X,\F)$ is pointwise asymptotically commutative
 \end{enumerate}
 we have, $$(1)\implies (2)\implies (3)$$
We now present an example demonstrating that the implications are strict, i.e., the converse of any implication does not hold in general.
 \end{remark}

 \begin{ex}
Let $X=[0,1]$ and define 
$f_n:[0,1]\rightarrow[0,1]$ as 
$$f_1(x)=
\begin{cases}
\frac{1}{2}(2x)^{\frac{1}{2}} & x\in[0,\frac{1}{2}],\\
\frac{1}{2}(2x-1)^{\frac{1}{2}}+\frac{1}{2}  & x\in[1/2,1].
\end{cases}
$$ 
and for $n\geq 2$
$$f_n(x)=
\begin{cases}
\frac{1}{2}(2x)^{\frac{n^2}{n^2-1}} & x\in[0,\frac{1}{2}],\\
\frac{1}{2}(2x-1)^{\frac{n^2}{n^2-1}}+\frac{1}{2}  & x\in[1/2,1].
\end{cases}
$$  

Let $([0,1],\F)$ be non autonomous system generated by sequence $(f_n)_{n\in \N}$. Also, it may be noted that the family $\F$ is non-commutative. Further we have,  
$$
\omega_n(x)=
\begin{cases}
\frac{1}{2}(2x)^{\frac{n}{n+1}} & x\in[0,\frac{1}{2}],\\
\frac{1}{2}(2x-1)^{\frac{n}{n+1}}+\frac{1}{2}  & x\in[1/2,1].
\end{cases}
$$  
Further note that, $\omega_n\to Id_{[0,1]}$ uniformly and consequently for every $r\in \N$
$$
\lim \limits_{n\rightarrow \infty}D_{C(X)}(\omega_n\circ \omega_r,\omega_r\circ\omega_n)=0.
$$
Hence $([0,1],\F)$ is asymptotically commutative.
 \end{ex}

 \begin{ex}

Let $\varphi_n : \mathbb{T}^1 \to \mathbb{R}$ be defined as

$$
\varphi_n(\theta) =
\begin{cases}
0 & \theta \notin [1/n, 2/n], \\[6pt]
\frac{1}{2} - \left|\theta - \frac{3}{2n}\right| \cdot n & \theta \in [1/n, 2/n].
\end{cases}
$$

Define a sequence $\{\omega_n\}$ of maps on $\mathbb{T}^2$ by
$$
\omega_n(\theta,\phi)=
\begin{cases}
(\theta,\ \phi+\varphi_n(\theta)) & \text{if } n \text{ is even},\\
(\theta+\varphi_n(\phi),\ \phi) & \text{if } n \text{ is odd},
\end{cases}
\quad (\mathrm{mod}\,1).
$$

\noindent Since each $\omega_n$ is a homeomorphism of $\mathbb{T}^2$, we can extract $f_n$ using the recursive identity $f_n=\omega_n\circ\omega_{n-1}^{-1}$. Let $(\mathbb{T}^2,\F)$ be non autonomous system generated by sequence $(f_n)_{n\in \N}$. Moreover, it may be noted that since the family $\{\omega_n\}$ is non-commutative, $\F$ is a non-commutative family. \\

\noindent Further note that, for each $(\theta,\phi)\in \mathbb{T}^2$, since $\varphi_n(x)\to 0$ for every $x\in \mathbb{T}^1$, we have
$$\omega_n(\theta,\phi)\to (\theta,\phi).$$
Thus $\omega_n \to \mathrm{Id}_{\mathbb{T}^2}$ pointwise and consequently for any $r$ and $(\theta,\phi)\in \mathbb{T}^2$, 

$$
d\big(\omega_n\circ \omega_r(\theta,\phi),\ \omega_r\circ \omega_n(\theta,\phi)\big)\to 0.
$$
Hence $(X,\F)$ is pointwise asymptotically commutative.\\

Now fix $r$ odd. Choose $\phi_0=\frac{3}{2r}$ and for each even $n$, choose $\theta_n = \frac{3}{2n} - \frac{1}{2} \ (\mathrm{mod}\,1).$ Then we have $\varphi_n(\theta_n)=0, $ and $\varphi_n(\theta_n+\tfrac{1}{2})=\tfrac{1}{2}.$
Further note that
$$d\big(\omega_n\circ \omega_r(\theta_n,\phi_0),\ \omega_r\circ \omega_n(\theta_n,\phi_0)\big)=\tfrac{1}{2}.$$
Thus, $\sup_{(\theta,\phi)} d\big(\omega_n\circ \omega_r,\ \omega_r\circ \omega_n\big)\ge \tfrac{1}{2}$
for infinitely many $n$. Therefore, the convergence is not uniform, and hence $(X,\F)$ is not asymptotically commutative.
 \end{ex}

\section{Invariance, Equicontinuity and Weak Mixing}\label{pinv} In this section, we derive some results for proximal relations for equicontinuous and weak mixing systems.

\begin{thm}\label{inv}
Let $(X,\F)$ be a non-autonomous system.
\begin{enumerate}
\item  If $(X,\mathbb{F})$ is pointwise asymptotically commutative, then the set of proximal (syndetically proximal) pairs is invariant.
\item  If $(X,\mathbb{F})$ is asymptotically commutative, then the set of regionally proximal (regionally syndetically proximal) pairs is invariant.
\end{enumerate}

\end{thm}
 
\begin{proof}
(1) Let $(x,y)$ be a proximal pair for $(X,\mathbb{F})$ and $r\in \mathbb{N}$ be fixed. As $(x,y)$ is proximal, there exists a sequence $(n_k)$ of integers such that $\lim \limits_{k\rightarrow \infty} d_X(\omega_{n_k}(x), \omega_{n_k}(y))=0$. Since $\omega_r$ is continuous, $\lim \limits_{k\rightarrow \infty} d_X(\omega_{r}(\omega_{n_k}(x)), \omega_{r}(\omega_{n_k}(y))=0$. Also as $(X,\F)$ is pointwise asymptotically commutative, we have $\lim \limits_{k\rightarrow \infty} d_X(\omega_{n_k}(\omega_r(x)), \omega_{n_k}(\omega_r(y))=0$ and thus $(\omega_r(x),\omega_r(y))$ is a proximal pair for $(X,\mathbb{F})$. Since every instance of proximality for $(x,y)$ gives an instance of proximality for $(\omega_r(x),\omega_r(y))$, the pair $(\omega_r(x),\omega_r(y))$ is proximal (respectively, syndetically proximal).\\

(2) Let $(x,y)$ be a regionally proximal pair for $(X,\mathbb{F})$ and $r\in \mathbb{N}$ be fixed. As $(x,y)$ is regionally proximal, there exists sequences $(x_n)$, $(y_n)$ in $X$ and $(k_n)$ in $\N$ such that $x_n\rightarrow x$, $y_n\rightarrow y$ and $d_X(\omega_{k_n}(x_n),\omega_{k_n}(y_n))\rightarrow 0$. As $\omega_r$ is continuous we have $\omega_r(x_n)\rightarrow \omega_r(x)$, $\omega_r(y_n)\rightarrow \omega_r(y)$ and $d_X(\omega_r\circ\omega_{k_n}(x_n),\omega_r\circ\omega_{k_n}(y_n))\rightarrow 0$. Further as  $(X,\mathbb{F})$ is asymptotically commutative, $D_{C(X)}(\omega_{k_n}\circ\omega_r,\omega_r\circ\omega_{k_n})\rightarrow 0$, and thus $d_X(\omega_{k_n}(\omega_r(x_n)),\omega_{k_n}(\omega_r(y_n)))\rightarrow 0$. Consequently, $(\omega_r(x),\omega_r(y))$ is a regionally proximal pair. Once again, as the instances of proximality for $(x,y)$ are instances of proximality for  $(\omega_r(x),\omega_r(y))$, if $(x,y)$ is a regionally syndetically proximal pair then the pair $(\omega_r(x),\omega_r(y))$ is a regionally syndetically proximal pair. As the proof holds for each $r\in \N$, regionally proximality (regionally syndetically proximality) for $(x,y)$ ensures the same for $(\omega_r(x),\omega_r(y))$.
\end{proof}

\begin{Remark}
The above result establishes the invariance of various forms of proximality for a pointwise asymptotic commutative (asymptotic commutative) non-autonomous system. However, the result holds under the strict assumption of pointwise asymptotic commutativity and does not hold when the assumption is dropped. We now give an example in support of our claim.
\end{Remark}

\begin{ex} \label{a}
Let $X=[0,1]$ be the unit interval and let $$f_1(x)=f(x)= \left\{
		\begin{array}{lr}
			0 &  : 0\leq x \leq \frac{1}{4}, \\
			{4x-1} & : \frac{1}{4}\leq x\leq \frac{1}{2} \\
             {-\frac{4}{3}x+\frac{5}{3}} & : \frac{1}{2}\leq x\leq \frac{3}{4} \\
             \frac{2}{3} &  : \frac{3}{4}\leq x \leq 1
		\end{array}
		\right\},$$ 
   $$f_n(x)=g(x) = \left\{
		\begin{array}{lr}
			2x &  : 0\leq x \leq \frac{1}{2}, \\
			{2-2x} & : \frac{1}{2}\leq x\leq 1
            
		\end{array}
		\right\} ~~\text{~~for~~}~~ n\geq 2.$$

It may be noted that for $r=1$, $d_X(\omega_n\circ \omega_1(\frac{1}{2}), \omega_1\circ \omega_n(\frac{1}{2}))=\frac{2}{3}\nrightarrow 0$. Consequently, $(X,\F)$ is not pointwise asymptotically commutative. Also, the pair $(0,\frac{1}{2})$ is proximal (regionally proximal/ syndetically proximal/ regionally syndetically proximal), but $(\omega_1(0),\omega_1(\frac{1}{2})) = (0,1)$ is not proximal at any of the instances (as 
$(\omega_n(0),\omega_n(1))=(0,\frac{2}{3})$ for all $n>1$). Consequently, the result does not hold in the absence of pointwise asymptotically commutativity (asymptotically commutative).
\end{ex}

\begin{prop}
For any non-autonomous dynamical system $(X,\F)$, $Q(X,\F)$ is a closed relation in $X\times X$.
\end{prop}

\begin{proof}
Let $(x_1,x_2)\in \overline{Q(X,\F)}$, $U \times V$ be a neighborhood of $(x_1,x_2)$ and $\epsilon>0$ be given. As $(x_1,x_2)\in \overline{Q(X,\F)}$, there exists a point $(z_1,z_2)\in Q(X,\F)$ such that $(z_1,z_2)\in U \times V$. As $(z_1,z_2)$ is a regionally proximal pair, there exists $(w_1,w_2)\in U \times V$ and $k\in \mathbb{N}$ such that $d_X(\omega_{k}(w_1),\omega_{k}(w_2))<\epsilon$. Thus, $(x_1,x_2)$ is a regional proximal pair, and the proof is complete.
\end{proof}

\begin{thm}
For any non-autonomous dynamical system $(X,\mathbb{F})$ generated by a family of homeomorphisms, $(X,\mathbb{F})$ is equicontinuous if and only if $Q(X,\F)=\triangle$, where $\triangle$ denotes the diagonal of the set $X\times X$.
\end{thm}
 
\begin{proof}
Let $(X,\mathbb{F})$ be equicontinuous, $(x,y)\in Q(X,\F)$ and $\epsilon>0$ be given. As $(X,\mathbb{F})$ is equicontinuous, there exists $\delta>0$ such that $d_X(u,v)<\delta$ implies $d_X(\omega_m(u),\omega_m(v))<\epsilon$ for all $u,v\in X,~~m\in \Z$. As the pair $(x,y)$ is regionally proximal, there exist sequences $(x_n)$ and $(y_n)$ converging to $x$ and $y$ respectively and a sequence  $(k_n)$ in $\N$ such that $d_X(\omega_{k_n}(x_n),\omega_{k_n}(y_n))$ converges to zero. Thus, there exists $n_0\in \N$ such that $d_X(\omega_{k_n}(x_n),\omega_{k_n}(y_n))<\delta$ for $n\geq n_0$ and hence $d_X(\omega_m(\omega_{k_n}(x_n)),\omega_m(\omega_{k_n}(y_n)))<\epsilon$ for all $m\in \Z, n\geq n_0$. As the argument holds for any $\epsilon>0$, $d_X(x_n,y_n)$ converges to $0$ and thus we have $x=y$. \\

Conversely, let $Q(X,\F)=\triangle$ and $x\in X$ be a point of sensitivity (with sensitivity constant $\delta>0$). As the system is sensitive at $x\in X$, for any $n\in \N$, there exists $y_n\in X$ and $k_n\in \N$ such that $d_X(x,y_n)<\frac{1}{n}$ and $d_X(\omega_{k_n}(x),\omega_{k_n}(y_n))>\delta$. Consequently, if $x_n=\omega_{k_n}(x)$ and $z_n=\omega_{k_n}(y_n)$ converge to $p$ and $q$ respectively (after passing through subsequences if required), then $d_X(\omega_{-k_n}(x_n),\omega_{-k_n}(z_n))= d_X(x,y_n)$ converges to zero. Thus, $(p,q)$ is a non-trivial pair in $Q(X,\F)$ (which is a contradiction) and thus $(X,\mathbb{F})$ is equicontinuous.
\end{proof}

\begin{prop}
For any weakly mixing non-autonomous system $(X,\mathbb{F})$, $Q(X,\F)=X\times X$.
\end{prop}

\begin{proof}
Let $(x_1,x_2)$ be any point in $X\times X$, $U \times V$ be a neighborhood of $(x_1,x_2)$, $x_0\in X$ and $U_n$ be $\frac{1}{n}$-neighborhood of $x_0$. As $(X,\mathbb{F})$ is weakly mixing, there exists $k_n\in \mathbb{N}$ such that $(\omega_{k_n}(U)\times \omega_{k_n}(V))\cap (U_n\times U_n) \neq \emptyset$ and thus there exists $(u_n,v_n)\in U\times V$ such that such that $d_X(\omega_{k_n}(u_n),\omega_{k_n}(v_n))<\frac{2}{n}$. As the argument holds for any $n\in\N$ and any $(x_1,x_2)\in X\times X$, we have $Q(X,\F)=X\times X$.
\end{proof}

\section{Systems Generated by Uniformly Converging Sequence}\label{uniform}

In this section, we investigate proximal relations for systems generated by a uniform convergent sequence of maps.

\begin{thm}  \label{L1}  
Let $(X,\F)$ and $(Y,\mathbb{G})$ be non-autonomous dynamical systems generated by $\F=\{f_n:n\in\N \}$ and $\mathbb{G}=\{g_n:n\in\N\}$ respectively such that $(g_n)$ converges uniformly to $g$. If $(x_1,x_2)\in L(X,\F)$ and $(y_1,y_2)\in L(X,\mathbb{G})$, then for every $\epsilon>0$, there exists a syndetic set $A$ such that $d_X(\omega_{k,X}(x_1),\omega_{k,X}(x_2))<\epsilon$ and $d_Y(\omega_{k,Y}(y_1),\omega_{k,Y}(y_2))<\epsilon$ for all $k\in A$.
\end{thm}
\begin{proof}
Let $(X,\F)$ and $(Y,\mathbb{G})$ be non-autonomous systems generated by $\F=\{f_n:n\in\N \}$ and $\mathbb{G}=\{g_n:n\in\N\}$ and let $(g_n)$ converge uniformly to $g$. Further, let $(x_1,x_2)\in L(X,\F)$, $(y_1,y_2)\in L(Y,\mathbb{G})$ and let $\epsilon>0$ be given. As $(x_1,x_2)\in L(X,\F)$, there exists a syndetic set $A_1\subseteq \N$ (say $M_1$-syndetic) such that $d_X(\omega_{k,X}(x_1),\omega_{k,X}(x_2))<\epsilon$ for all $k\in A_1$. As $g$ is continuous, there exists a $\delta>0$ such that $d_Y(x,y)<\delta\implies d_Y(g^i(x),g^i(y))<\frac{\epsilon}{3}$ for $i\in\{1,2,\ldots,M_1\}$ for all $x,y\in Y$. Also as $(g_n)$ converges uniformly to $g$, there exists $n_0\in \N$ such that $D_{C(Y)}(g_{n}^{n+i},g^i)<\frac{\epsilon}{3}$ for all $n \geq n_0$ and $i\in\{1,2,\ldots,M_1\}$. As $y_1$ and $y_2$ are syndetically proximal in $(Y,\mathbb{G})$, there exists a syndetic set $A_2\subseteq \N$ (say $M_2$-syndetic with $\min\limits_{a_2\in A_2}a_2 > n_0$) such that $d_Y(\omega_{k,Y}(y_1),\omega_{k,Y}(y_2))<\delta$ for each $k\in A_2$ and hence $d_Y(g^i(\omega_{k,Y}(y_1)),g^i(\omega_{k,Y}(y_2)))<\frac{\epsilon}{3}$ for all $i\in \{1,2,\ldots,M_1\}$ and $k\in A_2$. Thus we have $d_Y(\omega_{k+i,Y}(y_1),\omega_{k+i,Y}(y_2))\leq d_Y(\omega_{k+i,Y}(y_1),g^i(\omega_{k,Y}(y_1)))+ d_Y(g^i(\omega_{k,Y}(y_1)),g^i(\omega_{k,Y}(y_2)))+ d_Y(g^i(\omega_{k,Y}(y_2)),\omega_{k+i,Y}(y_2))<\epsilon$ for all $i\in \{1,2,\ldots,M_1\}$ and $k\in A_2$. Finally, as $d_X(\omega_{k+i,X}(x_1),\omega_{k+i,X}(x_2))<\epsilon$ for some $i\in \{1,2,\ldots M_1\}$ and $k\in A_2$, there exists a syndetic set $A\subseteq \N$ such that $d_X(\omega_{k,X}(x_1),\omega_{k,X}(x_2))<\epsilon$ and $d_Y(\omega_{k,Y}(y_1),\omega_{k,Y}(y_2))<\epsilon$ for any $k\in A$.
\end{proof}

\begin{cor}\label{lxf}
Let $(X,\F)$ and $(Y,\mathbb{G})$ be non-autonomous dynamical systems generated by $\F=\{f_n:n\in\N \}$ and $\mathbb{G}=\{g_n:n\in\N\}$ respectively. If $(g_n)$ converges uniformly to $g$, then $((x_1,x_2),(y_1,y_2))\in L(X,\F)\times L(Y,\mathbb{G}) \Leftrightarrow ((x_1,y_1),(x_2,y_2)) \in L(X\times Y, \F \times \mathbb{G})$.
\end{cor}

\begin{proof}
Firstly note that as proximality of two tuples in $X\times Y$ ensures proximality in each of the coordinates, $((x_1,y_1),(x_2,y_2)) \in L(X\times Y, \F \times \mathbb{G})$ ensures $((x_1,x_2),(y_1,y_2))\in L(X,\F)\times L(Y,\mathbb{G})$. Conversely, let $(g_n)$ converge uniformly to $g$, $((x_1,x_2),(y_1,y_2))\in L(X,\F)\times L(Y,\mathbb{G})$ and let $\epsilon>0$ be given. Then by Theorem \ref{L1}, there exists a syndetic set $A$ such that $d_X(\omega_{k,X}(x_1),\omega_{k,X}(x_2))<\epsilon$ and $d_Y(\omega_{k,Y}(y_1),\omega_{k,Y}(y_2))<\epsilon$. Thus $L(X,\F)\times L(Y,\mathbb{G}) \subseteq L(X\times Y, \F \times \mathbb{G})$ holds and the proof is complete.
\end{proof}

\begin{cor}
Let $(X,\F)$ and $(Y,\mathbb{G})$ be non-autonomous dynamical systems generated by $\F=\{f_n:n\in\N \}$ and $\mathbb{G}=\{g_n:n\in\N\}$ respectively. If $(g_n)$ converges uniformly to $g$, then $((x_1,x_2),(y_1,y_2))\in M(X,\F)\times M(Y,\mathbb{G}) \Leftrightarrow ((x_1,y_1),(x_2,y_2)) \in M(X\times Y, \F \times \mathbb{G})$.
\end{cor}

\begin{proof}
Let $(X,\F)$ and $(Y,\mathbb{G})$ be non-autonomous systems as given, $((x_1,x_2),(y_1,y_2))\in M(X,\F)\times M(Y,\mathbb{G})$ and $U_i\times V_i$ be neighborhood of $(x_i,y_i)$ in $X\times Y~~(i=1,2)$. As $U_1\times U_2$ and $V_1\times V_2$ are neighborhoods of $(x_1,x_2)$ and $(y_1,y_2)$ respectively, there exists $u_i\in U_i, v_i\in V_i$ ($i=1,2$) such that $(u_1,u_2)$ and $(v_1,v_2)$ are syndetically proximal (in $(X,\F)$ and $(Y,\mathbb{G})$ respectively). Consequently, $(u_i,v_i)\in U_i\times V_i$ ($i=1,2$) are syndetically proximal (by Corollary \ref{lxf}) and thus $((x_1,y_1),(x_2,y_2)) \in M(X\times Y, \F \times \mathbb{G})$.\\

Conversely, let $((x_1,y_1),(x_2,y_2)) \in M(X\times Y, \F \times \mathbb{G})$ and $U_1\times U_2$, $V_1\times V_2$ be neighborhoods of $(x_1,x_2)$ and $(y_1,y_2)$ respectively. As $U_i\times V_i$ is a neighborhood of $(x_i,y_i)$ ~$(i=1,2)$, there exists $(u_1,v_1)$ and $(u_2,v_2)$ in $U_1\times V_1$ and  $U_2\times V_2$ respectively such that $(u_1,v_1)$ and $(u_2,v_2)$ are syndetically proximal. Thus, $(u_1,u_2)$ and $(v_1,v_2)$ are syndetically proximal in $(X,\F)$ and $(Y,\mathbb{G})$ respectively (by Corollary \ref{lxf}) and the proof of converse is complete. 
\end{proof}

\begin{Remark}\label{lp1}
The above proof relates the proximal pairs for the product of non-autonomous systems with the proximal pairs in the individual non-autonomous component systems. It may be noted that if only one of the proximal pairs $(x_1,x_2)$ (or $(y_1,y_2)$) is syndetically proximal, a similar proof guarantees that the pairs $(x_1,y_1)$ and $(x_2,y_2)$ are arbitrarily close to each other infinitely often (not necessarily in a syndetic manner) and hence are proximal in the product system. The proof uses the fact that if $(x_1,x_2)$ is a syndetically proximal pair in $X$ and $(g_n)$ converges uniformly to $g$ then the pairs $(x_1,x_2)$ and $(y_1,y_2)$ are proximal at a common set of times and hence the points $(x_1,y_1)$ and $(x_2,y_2)$ are proximal for the product system. Also, as the above proof establishes that the common instances of proximality for any two syndetically proximal pairs are also syndetic, if the pairs $(x_1,x_2)$ and $(y_1,y_2)$ are regionally syndetically proximal, the pair $((x_1,y_1),(x_2,y_2))$ is regionally syndetically proximal in the product system. Thus, we get the following results.
\end{Remark}

\begin{thm}\label{lp}
Let $(X,\F)$ and $(Y,\mathbb{G})$ be non-autonomous dynamical systems generated by $\F=\{f_n:n\in\N \}$ and $\mathbb{G}=\{g_n:n\in\N\}$ respectively. If $(g_n)$ converges uniformly to $g$, then $((x_1,x_2),(y_1,y_2))\in L(X,\F)\times P(Y,\mathbb{G}) \Rightarrow ((x_1,y_1),(x_2,y_2)) \in P(X\times Y, \F \times \mathbb{G})$.
\end{thm}

\begin{proof}
Let $(g_n)$ converge uniformly to $g$, $((x_1,x_2),(y_1,y_2))\in L(X,\F)\times P(Y,\mathbb{G})$ and let $\epsilon>0$ be given. As $(x_1,x_2)\in L(X,\F)$, there exists a syndetic set $A\subseteq \N$ (say $M$-syndetic) such that $d_X(\omega_{k,X}(x_1),\omega_{k,X}(x_2))<\epsilon$ for all $k\in A$. As $g$ is continuous, there exists a $\delta>0$ such that $d_Y(x,y)<\delta\implies d_Y(g^i(x),g^i(y))<\frac{\epsilon}{3}$ for $i\in\{1,2,\ldots,M\}$ and for all $x,y\in Y$. Also, as $(g_n)$ converges uniformly to $g$, there exists $n_0\in \N$ such that $D_{C(Y)}(g_{n}^{n+i},g^i)<\frac{\epsilon}{3}$ for all $n>n_0$ and $i\in\{1,2,\ldots,M\}$. Further, as $y_1$ and $y_2$ are proximal in $(Y,\mathbb{G})$, there exists $k_0$ ($k_0>n_0$) such that $d_Y(\omega_{k_0,Y}(y_1),\omega_{k_0,Y}(y_2))<\delta$ and hence $d_Y(g^i(\omega_{k_0,Y}(y_1)),g^i(\omega_{k_0,Y}(y_2)))<\frac{\epsilon}{3}$ for all $i\in \{1,2,\ldots,M\}$. Consequently, $d_Y(\omega_{k_0+i,Y}(y_1),\omega_{k_0+i,Y}(y_2))\leq d_Y(\omega_{k_0+i,Y}(y_1),g^i(\omega_{k_0,Y}(y_1))) + d_Y(g^i(\omega_{k_0,Y}(y_1)),g^i(\omega_{k_0,Y}(y_2))) + d_Y(g^i(\omega_{k_0,Y}(y_2)),\omega_{k_0+i,Y}(y_2))<\epsilon$ for all $i\in \{1,2,\ldots,M\}$. Thus, as $d_X(\omega_{k_0+i,X}(x_1),\omega_{k_0+i,X}(x_2))<\epsilon$ for some $i\in \{1,2,\ldots M\}$, $((x_1,y_1),(x_2,y_2))$ is proximal for $X\times Y$ and the proof is complete. 
\end{proof}



\begin{cor}
Let $(X,\F)$ be a non-autonomous system generated by a family $\F = (f_n)$, if $(f_n)$ converges uniformly (to $f$), then $L(X,\F)$ is an equivalence relation.
\end{cor}

\begin{proof}
It may be noted that as proximality is reflexive and symmetric, it is sufficient to establish transitivity of $L(X,\mathbb{F})$ to establish the stated result. Also, if $(x,y)$ and $(y,z)$ are syndetically proximal pairs for $(X,\mathbb{F})$, proof of Theorem \ref{L1} establishes syndetic proximality of both the pairs under a common subsequence of $\mathbb{N}$ (say $(n_k)$). Consequently, for any $\epsilon>0$, there exists a syndetic set $A$ such that  $d_X(\omega_{n}(x),\omega_{n}(y)) < \frac{\epsilon}{2}$ and $d_X(\omega_{n}(y),\omega_{n}(z))<\frac{\epsilon}{2}$ hold for any $n\in A$. Thus, $d_X(\omega_{n}(x),\omega_{n}(z))<d_X(\omega_{n}(x),\omega_{n}(y))+ d_X(\omega_{n}(y),\omega_{n}(z))<\epsilon$ for all $n\in A$ and the pair $(x,z)$ is syndetically proximal. Consequently, $L(X,\mathbb{F})$ is a transitive relation and the proof is complete.
\end{proof}

\begin{Remark}
The above result establishes that if generating sequence for a non-autonomous system is uniformly convergent then the relation of syndetic proximality is an equivalence relation. As the result holds in the autonomous case, the above result is an analogous extension of a result known for the autonomous systems. However, it may be noted that uniform convergence plays an important role in establishing the above result and the result does not hold in the absence of uniform convergence. We now give an example in support of our claim.
\end{Remark}

\begin{ex}\label{sp}
Let $(X,f)=(\{0,1\}^{\N},\sigma)$ be the full shift and $f_n:X\rightarrow X$ be defined as $f_{2k}=f_{2k-1}=\sigma^{2^k}$ for $k\in N$. Then, for $x=0^\infty$, $y=1^20^21^40^4\ldots$, and $z=1^\infty$, as the pairs $(x,y)$ and $(y,z)$ approach arbitrarily close to each other along odd and even instances respectively, the pairs $(x,y)$ and $(y,z)$ are syndetically proximal. However, as the pair $(x,z)$ is not proximal, the result fails to hold in the absence of uniform convergence.
\end{ex}

\begin{cor}
Let $(X,\F)$ be a non-autonomous system generated by a family $\F = (f_n)$, if $(f_n)$ converges uniformly (to $f$), then $LP(X,\F)=PL(X,\F)=P(X,\F)$.
\end{cor}

\begin{proof}
Firstly note that any proximal pair $(x,y)$ can be viewed as $(x,x)(x,y)\in LP(X,\F)$ and $(x,y)(y,y)\in PL(X,\F)$, the relations  $P(X,\F)\subseteq LP(X,\F)$ and $P(X,\F)\subseteq PL(X,\F)$ hold trivially. Also, if $(x,y)\in L(X,\F)$ and $(y,z)\in P(X,\F)$, then the proof of Theorem \ref{lp} establishes proximality of $(x,y)$ and $(y,z)$ along a common subsequence and thus confirms proximality of $(x,z)$ (in $(X,\F)$). Thus $LP(X,\F)\subseteq P(X,\F)$ holds. Further, as a similar argument establishes that $(x,y)\in P(X,\F)$, $(y,z)\in L(X,\F)$ ensures proximality of the pair $(x,z)$, $LP(X,\F)=PL(X,\F)=P(X,\F)$ holds.
 \end{proof}

\begin{prop}
Let $(X,\F)$ be a non-autonomous system generated by a family $\F=(f_n)$, if $(f_n)$ converges uniformly to $f$ then $(x,y)\in L(X,\F)$ $\implies$ for every $\epsilon>0$, the set $\{n\in \N: d_X(\omega_n(x),\omega_n(y))<\epsilon\}$ is thickly syndetic.
\end{prop}

\begin{proof}
Let $(x,y)\in L(X,\F)$, $k\in\N$ and $\epsilon>0$ be given. Since $f_n$ converges uniformly to $f$, there exists $\delta>0$ and $N\in\N$ such that $d_X(x,y)<\delta\implies d_X(f_n^{n+i}(x),f_n^{n+i}(y))<\epsilon$ for all $n>N$ and $i\in\{1,2,\ldots,k\}$. As $(x,y)\in L(X,\F)$, there exists syndetic set $A\subseteq \N$ (with $min(A)>N$) such that $d_X(\omega_l(x),\omega_l(y))<\delta$ for all $l\in A$. Consequently, we have $d_X(\omega_{l+i}(x),\omega_{l+i}(y))<\epsilon$ whenever $i\in\{1,2,\ldots,k\}$, $l\in A$ and hence the set $\{n\in \N: d_X(\omega_n(x),\omega_n(y))<\epsilon\}$ is thickly syndetic. 
\end{proof}

\section{Systems Generated by Strongly Converging Sequence}\label{strongc}

In this section, we investigate proximal relations for systems generated by a strongly convergent sequence of maps.
\begin{thm}\label{L}
Let $(X, \F)$ be a non-autonomous system generated by a sequence $\F$ commuting with $f$. If $(f_n)$ converges strongly to $f$ then $(x,y)$ is syndetically proximal for $(X,f)$ $\implies$ $(x,y)$ is syndetically proximal for $(X,\F)$. Further, if $(X,\F)$ is pointwise asymptotically commutative and $\F$ is a family of homeomorphisms, then the converse is also true. 
\end{thm}

\begin{proof}
Let $(x,y)$ be syndetically proximal for $(X,f)$ and $\epsilon>0$ be given. As $\sum \limits_{i=1}^\infty D_{C(X)}(f_n, f)<\infty$, there exists $r\in \N$ such that $\sum \limits_{n=r}^\infty D_{C(X)}(f_n, f)<\frac{\epsilon}{3}$. Also, as $\omega_r$ is continuous, there exists $\delta>0$ such that $d_X(u,v)<\delta$ implies $d_X(\omega_r(u),\omega_r(v))<\frac{\epsilon}{3}$. Further, as the pair $(x,y)$ is syndetically proximal for $(X,f)$, there exists a syndetic set $A\subseteq \N$ (say M-syndetic) such that $d_X(f^{k}(x),f^{k}(y))<\delta$ for any $k\in A$ and thus $d_X(\omega_r(f^{k}(x)),\omega_r(f^{k}(y)))<\frac{\epsilon}{3}$. As $D_{C(X)}(\omega_{r+k},f^{k}\circ\omega_r)\leq \sum \limits_{j=1}^{k} D_{C(X)}(f_{r+j}, f)<\frac{\epsilon}{3}$, we have $d_X(\omega_{r+k}(x),\omega_{r+k}(y))<\epsilon$ for any $k\in A$ (by triangle inequality) and hence the pair $(x,y)$ is syndetically proximal for $(X,\F)$.\\

Conversely, let $(x,y)$ be syndetically proximal for $(X,\F)$ and $\epsilon>0$ be given. Since $\sum \limits_{i=1}^\infty D_{C(X)}(f_n, f)<\infty$, there exists $n_0\in \N$ such that $D_{C(X)}(f_n^{n+k},f^k)<\frac{\epsilon}{3}$ for all $k\in \N$, $n\geq n_0$. Also as $(X,\F)$ is asymptotically commutative and $\F$ is family of homeomorphisms, $(\omega_{n_0}^{-1}(x), \omega_{n_0}^{-1}(y))$ is syndetically proximal for $(X,\F)$ and thus there exists a syndetic set $A\subseteq\N$ (say $M$-syndetic and $\min \limits_{a\in A} a \geq n_0$) such that $d_X(\omega_{n_0+k}(\omega_{n_0}^{-1}(x)),\omega_{n_0+k}(\omega_{n_0}^{-1}(y)))<\frac{\epsilon}{3}$ for any $k\in A$. Thus we have $d_X(f^{k}(x),f^{k}(y))<\epsilon$ for any $k\in A$ and the proof is complete.
\end{proof}

\begin{Remark}\label{r3}
The above result relates the syndetic proximality of the non-autonomous system with the syndetic proximality of the limiting system. As the result uses the instances of proximality of one system to generate instances of proximality for the other, a similar proof establishes the equivalence of proximality for two systems under an identical set of assumptions. Further, as the system establishes equivalence of proximality of a pair in the two systems, a similar proof establishes equivalence of regional proximality under the same set of assumptions. As an autonomous system is proximal if and only if it is syndetically proximal, a non-autonomous system is proximal if and only if it is syndetically proximal (under the above stated assumptions). We now establish our claims below.
\end{Remark}

\begin{ex}\label{ex3}
Let $X=[0,1]$ be the unit interval and let $f_n:X\rightarrow X$ be defined as 
$$f_1(x)=f(x)= \left\{
		\begin{array}{lr}
			x &  : 0\leq x \leq \frac{1}{2}, \\
			\frac{4}{3}x-\frac{1}{6} & : \frac{1}{2}\leq x\leq  \frac{7}{8},\\
             1 & : \frac{7}{8}\leq x\leq 1
            
		\end{array}
		\right\},$$ and for $n>1$, $$f_n(x)=g(x) = \left\{
		\begin{array}{lr}
			-2x+\frac{1}{2} &  : 0\leq x \leq \frac{1}{4}, \\
			2x-\frac{1}{2} & : \frac{1}{4}\leq x\leq \frac{1}{2},\\
             x & : \frac{1}{2}\leq x\leq 1 
            
		\end{array}
		\right\}.$$

Let $(X, \F)$ be the non-autonomous system generated by $F = \{f, g, g, \ldots\}$. Then, $f$ and $g$ commute with each other and $\sum \limits_{i=1}^\infty D_{C(X)}(f_n, g)<\infty$ holds. However, as any pair of points $(x,y)$ ($x,y \in (\frac{7}{8},1))$ is syndetically proximal for $(X,\F)$ but fails to be syndetically proximal for $(X, g)$, the Theorem \ref{L} fails to hold if the members of generating sequence fail to be homeomorphisms.
\end{ex}

\begin{cor}\label{RP}
Let $(X, \F)$ be a non-autonomous system generated by a family $\F$ commuting with $f$. If $(f_n)$ converges strongly to $f$ then, $(x, y)$ is regionally proximal for $(X, f) \implies (x, y)$ is regionally proximal for $(X, \F)$. Further, if $(X,\F)$ is asymptotically commutative and $\F$ is a family of homeomorphisms, the converse is also true.
\end{cor}

\begin{proof}
The proof follows from discussions in Remark \ref{r3} and Theorem \ref{L}.
\end{proof}

\begin{cor}
Let $(X, \F)$ be a non-autonomous system generated by a family $\F$ commuting with $f$. If $(f_n)$ converges strongly to $f$ then, $(x, y)$ is regionally syndetically proximal for $(X, f) \implies (x, y)$ is regionally syndetically proximal for $(X, \F)$. Further, if $(X,\F)$ is asymptotically commutative and $\F$ is a family of homeomorphisms,  the converse also holds.
\end{cor}

\begin{proof}
The proof follows from discussions in Remark \ref{r3} and Theorem \ref{L}.
\end{proof}

\begin{cor}
Let $(X, \F)$ be a pointwise asymptotically commutative non-autonomous system generated by a family of homeomorphisms $\F$. If $(f_n)$ converges strongly to $f$ then, $(X,\F)$ is proximal $\iff$ $(X,\F)$ is syndetically proximal.
\end{cor}

\begin{proof}
If $(X,\F)$ be proximal then Remark \ref{r3} ensures that $(X,f)$ is proximal and hence syndetically proximal (Theorem $1$, \cite{SM}). Consequently, $(X,\F)$ is syndetically proximal (by Theorem \ref{L}) and the forward part holds. As any syndetically proximal pair is proximal, $(X,\F)$ is proximal $\iff$ $(X,\F)$ is syndetically proximal.
\end{proof}

\begin{Remark}\label{e3}
The above result relates the proximality of a non-autonomous system with the syndetic proximality of the same. In particular, the result establishes that if the system is generated by a sequence converging at a sufficiently fast rate, then the system is proximal if and only if it is syndetically proximal. However, if $f$ is a homeomorphism of $X$ such that $(X,f)$ is proximal then the system generated by $\F=\{f, f^{-1}, I, f^2, f^{-2}, I, I, f^3, f^{-3}, I, I, I, f^4,f^{-4}, \ldots\}$ is proximal but not syndetically proximal. Thus, the above result does not hold if uniform convergence of the generating family is dropped.
\end{Remark}

\begin{thm}
Let $(X, \F)$ be a non-autonomous system generated by a family $\F$ commuting with $f$. If If $(f_n)$ converges strongly to $f$ then, $(X,\F)$ is regionally proximal $\iff$ $(X,f)$ is regionally proximal.
\end{thm}

\begin{proof}
Let $(x,y)$ be a regionally proximal pair for $(X,\F)$ and $\epsilon>0$ be given. Let $U=S(x,\eta)$ and $V=S(y,\eta)$ be the $\eta$-neighborhoods of $x$ and $y$ respectively. As $\sum \limits_{n=1}^\infty D_{C(X)}(f_n, f)<\infty$, there exists $n_0\in\N$ such that $\sum \limits_{n=n_0}^\infty D_{C(X)}(f_n, f)<\frac{\epsilon}{3}$. As $\omega_{n_0}$ is continuous, there exists $\delta>0$ such that $d_X(a,b)<\delta \implies d_X(\omega_{n_0}(a),\omega_{n_0}(b))<\eta$. Then, for any $x_1\in \omega_{-n_0}(x)$, $y_1\in \omega_{-n_0}(y)$, as $(x_1,y_1)\in Q(X,\F)$, there exists $p\in S(x_1,\delta)$, $q\in S(y_1,\delta)$ and $i\in \N$ such that $d_X(\omega_{n_0+i}(p),\omega_{n_0+i}(q))<\frac{\epsilon}{3}$ and thus $d_X(f^i(\omega_{n_0}(p)),f^i(\omega_{n_0}(q)))<\epsilon$ (by triangle inequality). Also, as $d_X(p,x_1)<\delta$ ($d(q,y_1)<\delta$) we have $d_X(\omega_{n_0}(p),x)<\eta$ ($d_x(\omega_{n_0}(q),y)<\eta$) and thus $(\omega_{n_0}(p),\omega_{n_0}(q))\in U\times V$ such that $d_X(f^i(\omega_{n_0}(p)),f^i(\omega_{n_0}(q)))<\epsilon$. Thus $(x,y)\in Q(X,f)$ and the proof of forward part is complete.\\

As every regionally proximal pair for $(X,f)$ is regionally proximal for $(X,\F)$ (by Corollary \ref{RP}), the converse also holds.
\end{proof}

\begin{thm}
Let $(X, \F)$ be a pointwise asymptotically commutative non-autonomous system. If $(f_n)$ converges strongly to $f$ then,  $(x,y)\in L(X,\F) \implies \overline{\mathcal{O}_\F((x,y))}\subseteq L(X,\F)$. 
\end{thm}
\begin{proof}
Firstly note that if $(x,y)\in L(X,\F)$, then $\mathcal{O}_\F((x,y))\subseteq L(X,\F)$ (Theorem \ref{inv}). Further let $(z_1,z_2)\in \overline{\mathcal{O}_\F((x,y))}\setminus\mathcal{O}_\F((x,y)) $ and $\epsilon>0$ be given. As $L(X,f)\subseteq L(X,\F)$ (Theorem \ref{L}), it is sufficient to show $(z_1,z_2)$ is syndetical proximal pair for $(X,f)$. Now as $(x,y)\in L(X,\F)$, there exists a syndetic set $A\subset\N$ (say M-syndetic) such that $d_X(\omega_k(x),\omega_k(y))<\frac{\epsilon}{3}$ for all $k\in A$. Also as $(f_n)$ converges strongly to $f$, there exists $n_0\in \N$ such that $d(f_n^{n+j},f^j)<\frac{\epsilon}{3}$ for all $n\geq n_0$ and $j\in \N$. Now suppose there exists $k\in \N$ such that $d(f^{k+i}(z_1),f^{k+i}(z_2))>\epsilon$ for all $i\in \{1,\ldots M\}$. Then there exists $U_1$ and $U_2$ (neighborhood of $z_1$ and $z_2$ respectively) such that $d(f^{k+i}(p_1),f^{k+i}(p_2))>\epsilon$ for all $p_1\in U_1$, $p_2\in U_2$ and $i\in \{1,\ldots M\}$. As $(z_1,z_2)\in \overline{\mathcal{O}_\F((x,y))}\setminus\mathcal{O}_\F((x,y))$, there exists $m\geq n_0$ such that $(\omega_m(x),\omega_m(y))\in U_1\times U_2$ and thus $d(f^{k+i}(\omega_m(x)),f^{k+i}(\omega_m(y)))>\epsilon$ for all $i\in \{1,\ldots M\}$. Now it may be noted that $d(f^{k+i}(\omega_m(x)),f^{k+i}(\omega_m(y)))<d(f^{k+i}(\omega_m(x)),f_m^{m+k+i}(\omega_m(x)))+d(f_m^{m+k+i}(\omega_m(x)),f_m^{m+k+i}(\omega_m(y)))+d(f_m^{m+k+i}(\omega_m(y)),f^{k+i}(\omega_m(y)))$. Since $d(f^{k+i}(\omega_m(x)),f_m^{m+k+i}(\omega_m(x)))<\frac{\epsilon}{3}$ and $d(f_m^{m+k+i}(\omega_m(y)),f^{k+i}(\omega_m(y)))<\frac{\epsilon}{3}$ we have $d(\omega_{m+k+i}(x),\omega_{m+k+i}(y))>\frac{\epsilon}{3}$ for all $i\in\{1,2,\ldots, M\}$ which is a contradiction. Therefore, $(z_1,z_2)$ is syndetical proximal pair for $(X,f)$ (and consequently for $(X,\F)$) and $\overline{\mathcal{O}_\F((x,y))}\subseteq L(X,\F)$ holds.
\end{proof}

\begin{thm}\label{p17}
Let $(X, \F)$ be a pointwise asymptotically commutative non-autonomous system generated by a sequence of homeomorphisms $(f_n)$ converging strongly to $f$. Then, $(x,y)\in P(X,\F)$ such that $\overline{\mathcal{O}_\F((x,y))}\subseteq P(X,\F)$ then $(x,y)\in L(X,\F)$. 
\end{thm}

\begin{proof}
Let $(X, \F)$ be a a pointwise asymptotically commutative non-autonomous system generated by a family of homeomorphisms $\F$ and $(x,y)\in P(X,\F)$ such that $\overline{\mathcal{O}_\F((x,y))}\subseteq P(X,\F)$. Then by Remark \ref{r3}; $\overline{\mathcal{O}_\F((x,y))}\subseteq P(X,f)$. Let $\epsilon>0$ be given. For any $z=(z_1,z_2)\in \overline{\mathcal{O}_\F((x,y))}$, there exists $n_z\in \N$ such that $d(f^{n_z}(z_1),f^{n_z}(z_2))<\frac{\epsilon}{3}$ and thus there exists a neighborhood $U_{z_1}\times V_{z_2}$ of $(z_1,z_2)$ such that $d(f^{n_z}(p_1),f^{n_z}(p_2))<\frac{\epsilon}{3}$ for all $(p_1,p_2)\in U_{z_1}\times V_{z_2}$. Since $\overline{\mathcal{O}_\F((x,y))}$ is compact, there exists $m\in \N$ and $\{w_i=(a_i,b_i): i\in \{1,2,\ldots,m\}\}\subseteq \overline{\mathcal{O}_\F((x,y))}$ such that $\overline{\mathcal{O}_\F((x,y))}\subseteq \bigcup\limits_{i=1}^{m} (U_{a_i}\times V_{b_i})$. Also, as $(f_n)$ converges strongly to $f$, there exists $n_0\in \N$ such that $d(f_n^{n+i},f^i)<\frac{\epsilon}{3}$ for all $n\geq n_0$ and $i\in \N$. Let $k\geq n_0$ be fixed. Then, $(\omega_{k}(x),\omega_k(y))\in U_{a_j}\times V_{b_j}$ for some $j\in \{1,\ldots m\}$ and thus $d(f^{n_{w_j}}(\omega_{k}(x)),f^{n_{w_j}}(\omega_{k}(y)))<\frac{\epsilon}{3}$. Therefore, we have $d(\omega_{k+n_{w_j}}(x),\omega_{k+n_{w_j}}(y)) \leq d(\omega_{k+n_{w_j}}(x),f^{n_{w_j}}(\omega_k(x)))+ d(f^{n_{w_j}}(\omega_k(x)),f^{n_{w_j}}(\omega_k(y)))+d(f^{n_{w_j}}(\omega_k(y)),\omega_{k+n_{w_j}}(y))<\epsilon$ and thus $(x,y)$ is syndetically proximal with syndetic bound $M=\max \{n_0, \max\limits_{1\leq i\leq m}\{n_{w_i}\}\}$. Thus, the pair $(x,y)$ is syndetically proximal for $(X,\F)$, and the proof is complete.
 \end{proof}

\begin{cor}\label{p90}
Let $(X, \F)$ be a pointwise asymptotically commutative non-autonomous system generated by a sequence of homeomorphisms $(f_n)$ converging strongly to $f$. If $P(X,\F)$ is closed, then $P(X,\F)= L(X,\F)$ and hence is an equivalence relation.
\end{cor}

\begin{proof}
The proof follows from that fact that if $(x,y)\in P(X,\F)$ and $P(X,\F)$ is closed then $\overline{\mathcal{O}_\F((x,y))}\subseteq P(X,\F)$ and hence $(x,y)\in L(X,\F)$ (by Theorem \ref{p17}). Thus, every proximal pair is syndetically proximal, and the proof is complete.
\end{proof}

\begin{cor}
Let $(X, \F)$ be a pointwise asymptotically commutative non-autonomous system generated by a commutative family of homeomorphisms $\F$. If $(f_n)$ converges strongly to $f$ then, $Q(X,\F)\subseteq P(X,\F)$ $\iff$ $L(X,\F)=M(X,\F)=Q(X,\F)=P(X,\F)$.
\end{cor}
\begin{proof}
Since $Q(X,\F)\subseteq P(X,\F)$, $P(X,\F)=Q(X,\F)$ and consequently $P(X,\F)$ is closed and hence by Corollary \ref{p90}, $P(X,\F)=L(X,\F)$. Also we have $L(X,\F)\subseteq M(X,\F)\subseteq Q(X,\F)$ and hence $L(X,\F)=M(X,\F)=Q(X,\F)=P(X,\F)$.
\end{proof}

\section*{Author Contributions}
Sushmita Yadav: Conception and design of study, Writing original draft, writing review, and editing.\\

Puneet Sharma: Conception and design of the study, writing review, and editing.

\section*{Funding}
This work is supported by the MoE (India) for financial support and the National Board for Higher Mathematics (NBHM) Grant No. 02011/27/2023NBHM(R.P)/R \& D II/6282 for financial support.

\section*{Declarations}

\textbf{Conflict of interest:}  The authors declare that they have no conflict of interest.

\end{sloppypar}
\end{document}